\documentclass[12pt]{amsart}

\usepackage[left=2.3cm,right=2.3cm]{geometry}
\usepackage[english]{babel} 
\usepackage[utf8]{inputenc}
\usepackage{amsmath, amsfonts, amssymb,amsthm,graphicx}
\usepackage{tikz-cd}
\usepackage{hyperref}

\newtheorem{theorem}{Theorem}[section]
\newtheorem{Lemma}[theorem]{Lemma}
\newtheorem{Prop}[theorem]{Proposition}

\newcommand{\Fq}[1][]{\mathbb{F}_{q^{#1}}}
\newcommand{\Z}{\mathbb{Z}}
\newcommand{\Q}{\mathbb{Q}}

\DeclareMathOperator{\End}{End}
\DeclareMathOperator{\Cof}{Cof}
\DeclareMathOperator{\ICM}{ICM}

\begin{document}
\title{Cyclic abelian varieties over finite fields in ordinary isogeny classes}
\author{Alejandro J. Giangreco-Maidana}
\address{Aix Marseille University, CNRS, Centrale Marseille, I2M UMR 7373, 13453 Marseille, France \\
ORCID: https://orcid.org/0000-0002-1521-3266 }
\email{ajgiangreco@gmail.com}
\date{\today}

\begin{abstract}
Given an abelian variety $A$ defined over a finite field $k$, we say that $A$ is \emph{cyclic} if its group $A(k)$ of rational points is cyclic. In this paper we give a bijection between cyclic abelian varieties of an ordinary isogeny class $\mathcal{A}$ with Weil polynomial $f_{\mathcal{A}}$ and some classes of matrices with integer coefficients and having $f_{\mathcal{A}}$ as characteristic polynomial.
\end{abstract}

\maketitle

\section{Introduction}
The group $A(k)$ of rational points of an abelian variety $A$ defined over a finite field $k=\Fq$ is a finite abelian group, and it has theoretical and practical interests. More precisely, the group structure of $A(k)$ and some statistics are research topics in the literature. 

The structure of all possible groups for elliptic curves defined over finite fields was independently discovered in \cite{SCHOOF1987183}, \cite{ruck1987note},\cite{tsfasman1985group} and \cite{voloch1988note}. For higher dimensions, Rybakov gives in \cite{Rybakov2010} a very explicit description of all possible groups of rational points of an abelian variety in a given isogeny class, his result is formulated in terms of the characteristic polynomial of the isogeny class.

Cyclic subgroups of the group $A(k)$ of rational points of an abelian variety $A$ defined over a finite field $k$ are suitable for multiple applications, in particular for cryptography, where the Discrete Logarithm Problem is exploited.
We say that an abelian variety $A$ is \textbf{cyclic} if its group $A(k)$ of rational points is cyclic. Vl{\u{a}}du{\c{t}} studied the cyclicity of elliptic curves in \cite{VLADUT199913} and \cite{VLADUT1999354}. The higher dimensional case was studied by the author in \cite{GIANGRECOMAIDANA2019139} and \cite{GIANGRECOMAIDANA2020101628}, when varieties are grouped in their isogeny classes.

Deligne's functor describes an equivalence of categories between ordinary abelian varieties and modules over $\Z$ with certain properties. A classical result of Latimer and MacDuffee gives a bijection between certain classes
of matrices with integer coefficients and certain classes of fractional ideals. Combining these two results with a criterion for cyclic varieties based on their endomorphisms rings, we obtain our main result. 

Apart from theoretical interests, our result could be interesting from a computational point of view.

\section{Preliminaries and Statement of the result}
\subsection*{Abelian varieties over finite fields}
For the general theory of abelian varieties see for example \cite{mumford1970abelian}, and for precise results over finite fields, see \cite{Waterhouse1969}.

Let $q=p^r$ be a power of a prime, and let $k=\Fq$ be a finite field with $q$ elements. Let $A$ be an abelian variety of dimension $g$ over $k$. For an extension field $K$ of $k$, we denote by $\End_K (A)$ the ring of $K$-endomorphism of $A$ and by $\End^0_K (A) =(\End_K (A))\otimes\mathbb{Q}$ its endomorphism algebra, the latter being an invariant of its isogeny class $\mathcal{A}$, we can denote it by $\End^0_K(\mathcal{A})$.
For an integer $n$, denote by $A[n]$ the group of $n$-torsion points of $A(\overline{k})$. It is known that
\begin{align}
\begin{split}\label{eqn:torsion_points}
A[n]&\cong (\Z/n\Z)^{2g}, \qquad p\nmid n\\
A[p]&\cong (\Z/p\Z)^i, \qquad 0\leq i \leq g.
\end{split}
\end{align}

For a fixed prime $\ell$ ($\neq p$), the $A[\ell^n]$ form an inverse limit system under $A[n+1] \overset{\ell}{\rightarrow} A[n]$, and we can define \textbf{the Tate module} $T_\ell(A)$ by $\varprojlim A[\ell^n](\overline{k})$. This is a free $\Z_\ell$-module of rank $2g$ and the absolute Galois group $\mathcal{G}$ of $\overline{k}$ over $k$ operates on it by $\Z_\ell$-linear maps.

The Frobenius endomorphism $F$ of $A$ acts on $T_\ell(A)$ by a semisimple linear operator, and its characteristic polynomial $f_{A}(t)$ is called \textbf{Weil polynomial of} $A$ (also called \textbf{characteristic polynomial of} $A$). The Weil polynomial is independent of the choice of the prime $\ell$. Tate proved in \cite{tate1966} that an isogeny class $\mathcal{A}$ of abelian varieties is determined by its characteristic polynomial $f_{\mathcal{A}}(t)$. If $\mathcal{A}$ is simple, $f_{\mathcal{A}}(t)=h_{\mathcal{A}}(t)^e$ for some irreducible polynomial $h_{\mathcal{A}}$ and the center of $\End^0_k(\mathcal{A})$ is isomorphic to the number field $\mathbb{Q}(F)\cong \mathbb{Q}[t]/ (h_{\mathcal{A}}(t))$.
The cardinality of the group $A(k)$ of rational points of $A$ equals $f_{\mathcal{A}}(1)$, and thus, it is an invariant of the isogeny class. 

An abelian variety $A$ is \textbf{ordinary} if one of the following equivalent conditions is satisfied:
\begin{enumerate}
\item
$A$ has $p^g$ points of order dividing $p$ with coefficients in $\overline{k}$;
\item
The neutral component of the group scheme $A_p$, the kernel of multiplication by $p$, is of multiplicative type;
\item
At least half of the roots of $f_A$ are $p$-adic unities.
\end{enumerate}
The ordinariness is an invariant of the isogeny class.

\subsection*{Matrices}
Let us define some more notations. As usual, we denote by $M_n(\mathbb{Z})$ the set of square matrices of dimension $n\times n$ with integer entries. We define the \textbf{conjugacy classes of matrices} $\mathfrak{Cl}(M_n(\mathbb{Z}))$ as the quotient $M_n(\mathbb{Z})/\sim$ given by the equivalence relation $A\sim B$ if and only if $A=UBU^{-1}$ for some $U\in GL_n(\mathbb{Z})$, where $GL_n(\mathbb{Z})$ denotes the subset of $M_n(\mathbb{Z})$ of invertible matrices. 

For a polynomial $f$, we denote by $M_{n,f}(\mathbb{Z})$ the set of matrices in $M_n(\mathbb{Z})$ with $f$ as characteristic polynomial. In a similar way, we define $\mathfrak{Cl}(M_{n,f}(\mathbb{Z}))$, it is easy to see that this is well defined.

For a matrix $A\in M_n(\mathbb{Z})$, $\gcd(A)$ is the greatest common divisor of all entries of $A$ and the cofactor $\Cof(A)$ of $A$ is the matrix whose $ij$-entry is $(-1)^{i+j}$ times the determinant of the matrix that results from the elimination of the $i$-th row and the $j$-th column of $A$. 
Let us define the following map

\begin{center}
\begin{tabular}{c c c l}
$\tau$: &   $M_n(\mathbb{Z})$ &$\rightarrow$ &$\mathbb{Z}$ \\
    & $A$ &$\mapsto$ &$\gcd(\Cof(A))$.
\end{tabular}
\end{center}

The map $\tau$ can be extended to $\mathfrak{Cl}(M_n(\mathbb{Z}))$ since $\gcd(AU)=\gcd(A)$ if $\det (U)=\pm 1$ and $\Cof(AB)=\Cof(B)\Cof(A)$. 

\subsection*{Statement of the result}
If we denote by $\mathfrak{nc}(\mathcal{A})$ the subset of not cyclic varieties in an isogeny class $\mathcal{A}$, then our main result states:

\begin{theorem}\label{thm:bijection}
Let $\mathcal{A}$ be a $g$-dimensional isogeny class of ordinary simple abelian varieties. Let $f=f_{\mathcal{A}}$ be the Weil polynomial of $\mathcal{A}$. Then we have bijections between $\mathcal{A}$ and $\mathfrak{m}_{f,1}$, and between $\mathfrak{nc}(\mathcal{A})$ and $\mathfrak{m}_{f,2}$, where
\[
\mathfrak{m}_{f,c}:=\{ M\in \mathfrak{Cl}(M_{2g,f}(\mathbb{Z})),\; q^{g-1}|\tau(M),\; (\tau(1-M),f(1))\geq c\}.  
\]
\end{theorem}

\section{The proof}
Basically, the proof uses a version of the Deligne's functor and a classical result of Latimer and MacDuffee. Even if we will not use the classical version of Deligne's category, we will describe it later.

\subsection*{Fractional ideals}
Given a number field $K$, an \textbf{order} $\mathcal{O}$ in $K$ is a ring of $K$ which is finitely generated as a $\Z$-module and such that its fields of fractions equals $K$. The ring of integers $\mathcal{O}_K$ of $K$ is the maximal order of $K$. Given an algebraic integer $\theta$, $\Z[\theta]$ is an order in $\Q(\theta)$. The definition of order can be extended to finite-dimensional algebras over $\Q$.
Given an order $\mathcal{O}$ in a number field $K$, a \textbf{fractional $\mathcal{O}$-ideal} (or a \textbf{fractional ideal of} $\mathcal{O}$) is a nonzero finitely generated $\mathcal{O}$-module of $K$. Every fractional $\mathcal{O}$-ideal can be written as $\alpha\mathfrak{a}$, where $\mathfrak{a}$ is an (integral) ideal of $\mathcal{O}$ and $\alpha\in K^*$. Observe that every fractional ideal is also a finitely generated $\Z$-module. Denote by $I(\mathcal{O})$ the set of all fractional $\mathcal{O}$-ideals. Given two $\mathcal{O}$-fractional ideals $\mathfrak{a}$ and $\mathfrak{b}$, the product $\mathfrak{a}\mathfrak{b}$, the sum $\mathfrak{a}+\mathfrak{b}$, the intersection $\mathfrak{a}\cap \mathfrak{b}$, and the ideal quotient
\[
    (\mathfrak{a}:\mathfrak{b}) :=\{\alpha\in K: \alpha \mathfrak{b}\subset \mathfrak{a} \}
\]
are fractional $\mathcal{O}$-ideals. Note that if we have orders $\mathcal{O}' \subset \mathcal{O} \subset \mathcal{O}''$ and $\mathfrak{a}$ is a fractional $\mathcal{O}$-ideal, then $\mathfrak{a}$ is a fractional $\mathcal{O}'$-ideal, and if $\mathfrak{a}\mathcal{O}''\subset \mathfrak{a}$, then $\mathfrak{a}$ is also a fractional $\mathcal{O}''$-ideal. We have that $(\mathfrak{a}:\mathfrak{a})$ is a ring, and it is called the \textbf{multiplicator ring of} $\mathfrak{a}$. It is an order in $K$, and it is the biggest order such that $\mathfrak{a}$ is a fractional $(\mathfrak{a}:\mathfrak{a})$-ideal.

A \textbf{principal} $\mathcal{O}$-ideal is a fractional $\mathcal{O}$-ideal of the form $\lambda \mathcal{O}$ for some $\lambda\in K^*$. The \textbf{ideal class monoid} $\ICM(\mathcal{O})$ of an order $\mathcal{O}$ is the set $I(\mathcal{O})$ quotiented by the set $P(\mathcal{O})$ of principal fractional ideals of $\mathcal{O}$. One verifies easily that $\ICM(\mathcal{O})$ is well defined. We say simply that two fractional $\mathcal{O}$-ideals $\mathfrak{a}$ and $\mathfrak{b}$ are equivalent if they represent the same element in $\ICM(\mathcal{O})$, i.e. $\mathfrak{a}= (\alpha\mathcal{O})\mathfrak{b}$ for some $\alpha\in K^*$. Note that $(\mathfrak{a}:\mathfrak{a})=(\mathfrak{b}:\mathfrak{b})$ provided that $\mathfrak{a}$ is equivalent to $\mathfrak{b}$.

\subsection*{Deligne Category}
By Deligne functor (see its original paper \cite{deligne1969varietes} for details), we have an equivalence between the category of ordinary abelian varieties over $\Fq$ and some kind of modules over $\Z$. We will describe now this functor.

The \textbf{Deligne category} $\mathcal{L}_q$ is the category of pairs $(T,F)$, where $T$ is a finitely generated free module over $\mathbb{Z}$ and $F$ is an endomorphism of $T$ satisfying the following conditions:
\begin{enumerate}
\item
the endomorphism $F\otimes \mathbb{Q}$ of $T\otimes \mathbb{Q}$ is semisimple and its eigenvalues have absolute value $\sqrt{q}$,
\item
at least half of the roots of the characteristic polynomial of $F$ in $\overline{\mathbb{Q}}_p$, counting multiplicity, are $p$-adic units,
\item
there is an endomorphism $V$ of $T$ such that $FV=q$.
\end{enumerate}
A morphism of two objects $(T,F)$ and $(T',F')$ is a homomorphism $\varphi:T\rightarrow T'$ of $\mathbb{Z}$-modules such that $\varphi\circ F=F'\circ \varphi$.

We will now define a functor from the category of abelian varieties over a finite field $\overline{\mathbb{F}}_q$ to the Deligne category $\mathcal{L}_q$. Let $W=W(\overline{\mathbb{F}}_q)$ be the ring of Witt vectors over $\overline{\mathbb{F}}_q$ and let $\epsilon: W \rightarrow \mathbb{C}$ be an embedding of $W$ into the complex numbers that we now fix. Let $A$ be an abelian variety over $\Fq$ and let $A^{\#}$ be the Serre-Tate canonical lifting of $A$ to $W$ and let $F^{\#}$ be the lifting of the Frobenius of $A$ to $A^{\#}$. Put $T(A) = H_1(A^{\#}\otimes_{\epsilon} \mathbb{C})$ and denote by $F(A)$ the endomorphism of $T(A)$ induced by $F^{\#}$. The \textbf{Deligne functor} is defined by $A \mapsto (T(A),F(A))$.
\begin{theorem}[Deligne, 1969, \cite{deligne1969varietes}]
The Deligne functor $A \mapsto (T(A),F(A))$ is an equivalence of categories between the category of ordinary abelian varieties over $\Fq$ and the Deligne category $\mathcal{L}_q$.
\end{theorem}

The Deligne's equivalence is explicit in a convenient way using the language of fractional ideals:
\begin{theorem}[Deligne]\label{Th:Deligne}
Let $\mathcal{A}$ be an ordinary simple isogeny class of abelian varieties defined over $\Fq$, which is defined by the $q$-Weil polynomial $f_{\mathcal{A}}$. Let $\pi$ be a root of $f_{\mathcal{A}}$ corresponding to the Frobenius. Then
\begin{enumerate}
\item we have a bijection between $\mathcal{A}$ (up to $\Fq$-isomorphism) and $\ICM(\Z[\pi, q/\pi])$;
\item let $I_A$ be the corresponding fractional ideal of an abelian variety $A\in \mathcal{A}$, then $\End_{\Fq}(A)$ corresponds to the multiplicator ring $(I_A : I_A)$. 
\end{enumerate}
\end{theorem}
\begin{proof}
See for example \cite[Thm. 4.3.2]{marseglia2016thesis}
\end{proof}

Under this bijection, our next goal is to establish a connection between varieties and certain classes of matrices.

\subsection*{Latimer and MacDuffee}
The next step is to relate fractional ideals and matrices. The following is a classical result of Latimer and MacDuffee.
\begin{theorem}[Latimer and MacDuffee, 1933, \cite{LatimerICandMC1933}]\label{th:LatimerMacDuffee}
Let $f(T) \in \mathbb{Z}[T]$ be monic irreducible of degree $n$, and $\alpha$ is a root of $f(T)$. Then we have a bijection
\[
\ICM(\Z[\alpha]) \longleftrightarrow \mathfrak{Cl}(M_{n,f}(\Z)).
\]
\end{theorem}
\begin{proof}
Here we follow the proof of \cite{KConradICandMC}.

For any fractional $\Z[\alpha]$-ideal $\mathfrak{a}$ in $\mathbb{Q}(\alpha)$, multiplication by $\alpha$ is a $\Z$-linear map $m_\alpha : \mathfrak{a} \rightarrow \mathfrak{a}$. 
Since $\mathfrak{a}$ has a basis of size $n$ as a $\Z$-module, after choosing a $\Z$-basis, we can  represent $m_\alpha$ by a matrix $[m_\alpha] \in M_n(\Z)$. 
Changing the $\Z$-basis of $\mathfrak{a}$ changes the matrix representation of $m_\alpha$ to a conjugate matrix. So independent of a choice of basis we can associate to a fractional ideal $\mathfrak{a}$ the conjugacy class in $M_n(\Z)$ of a matrix representation for $m_\alpha : \mathfrak{a} \rightarrow \mathfrak{a}$. All matrices $A$ in this conjugacy class satisfy $f(A) = 0$ since $f(A) = f([m_\alpha]) = [m_{f(\alpha)}] = [m_0] = 0$.

For an equivalent fractional $\Z[\alpha]$-ideal $\mathfrak{b} = x\mathfrak{a}$, where $x \in \mathbb{Q}(\alpha)^*$, its conjugacy class of matrices (the matrices representing $m_\alpha : \mathfrak{b} \rightarrow \mathfrak{b}$ with respect to $\Z$-bases of $\mathfrak{b}$) is the same as that for $\mathfrak{a}$, since the matrix for $m_\alpha$ with respect to a $\Z$-basis $\{e_1, \dots, e_n\}$ of $\mathfrak{a}$ is the same matrix as that for $m_\alpha$ with respect to the $\Z$-basis $\{xe_1, \dots, xe_n\}$ of $\mathfrak{b}$. Thus we have a well-defined function
\begin{align}\label{BijDefICMC}
\ICM(\Z[\alpha])  \rightarrow \mathfrak{Cl}(M_{n,f}(\mathbb{Z})),
\end{align}
by the rule: pick a fractional ideal in the ideal class, pick a $\Z$-basis of it, write a matrix representation for $m_\alpha$ in terms of this basis, and use the conjugacy class of that matrix. We will show this function from fractional ideal classes to conjugacy classes of matrices is a bijection.

To show surjectivity, for every $A\in M_n(\Z)$ satisfying $f(A) = 0$ we will find a fractional $\Z[\alpha]$-ideal $\mathfrak{a}$ in $\mathbb{Q}(\alpha)$ such that $A$ is the matrix representation for $m_\alpha : \mathfrak{a} \rightarrow \mathfrak{a}$ with respect to some $\Z$-basis of $\mathfrak{a}$. Let $K=\mathbb{Q}(\alpha)=\mathbb{Q}[\alpha]$. Make $\mathbb{Q}^n$ into a $K$-vector space in the following way. For $c\in K$, write $c = g(\alpha)$ for $g(T) \in \mathbb{Q}[T]$. For $v\in \mathbb{Q}^n$, set
\begin{align}\label{eqEV_ICMC}
cv=g(\alpha)v:=g(A)v.
\end{align}
This is well-defined: if $c = h(\alpha)$ for $h(T)\in \mathbb{Q}[T]$ then $g(\alpha) = h(\alpha)$, so $g(T)-h(T)$ is divisible by $f(T)$ (because $f$ is the minimal polynomial of $\alpha$ in $\mathbb{Q}[T]$, as it is monic irreducible with root $\alpha$) and therefore $g(A) = h(A)$ as matrices (since $f(A) = 0$), so $g(A)v = h(A)v$ for all $v \in \mathbb{Q}^n$. If $v\in \Z^n$ then $\alpha v = Av$ is in $\Z^n$ since $A$ has integral entries, so $\Z^n$ is a $\Z[\alpha]$-submodule of $\mathbb{Q}^n$ that is finitely generated since $\Z^n$ is already finitely generated as an abelian group. From the way we define $\mathbb{Q}^n$ as a $K$-vector space, the equation $\alpha v = Av$ tells us the matrix representation of $m_\alpha$ on $\Z^n$ with respect to the standard basis of $\Z^n$ is $A$.

Treating $\mathbb{Q}^n$ as both a $\mathbb{Q}$-vector space and as $K$-vector space (by (\ref{eqEV_ICMC})), we have
\begin{align*}
n=\dim_{\mathbb{Q}} \mathbb{Q}^n = [K:\mathbb{Q}]\dim_K \mathbb{Q}^n = n \dim_K \mathbb{Q}^n,
\end{align*}
so $\mathbb{Q}^n$ is $1$-dimensional as a $K$-vector space. That means for any nonzero $v_0 \in \mathbb{Q}^n$, the $K$-linear map $\varphi_{v_0}: K \rightarrow \mathbb{Q}^n$ given by $\varphi_{v_0}(c) = cv_0$ is an isomorphism of $1$-dimensional $K$-vector spaces. The inverse image $\varphi_{v_0}^{-1}(\Z^n)$ is a finitely generated $\Z[\alpha]$-submodule of $K$ since $\Z^n$ has these properties inside $\mathbb{Q}^n$. So $\varphi_{v_0}^{-1}(\Z^n)$ is a fractional $\Z[\alpha]$-ideal in K. Call it $\mathfrak{a}$, so
\begin{align*}
\mathfrak{a}=\varphi_{v_0}^{-1}(\Z^n)=\{ c\in K, cv_0 \in \Z^n \}.
\end{align*}
Since $A$ is the matrix representation of $m_\alpha$ on $\Z^n$ with respect to its standard basis $\{e_1, \dots, e_n\}$, $A$ is also the matrix representation of $m_\alpha$ on $\mathfrak{a}$ with respect to the $\Z$-basis $\{\varphi_{v_0}^{-1}(e_1), \dots, \varphi_{v_0}^{-1}(e_n)\}$ of $\mathfrak{a}$. We have realized $A$ as a matrix representation for $m_\alpha$ on a fractional $\Z[\alpha]$-ideal $\mathfrak{a}$, so (\ref{BijDefICMC}) is onto.

To show (\ref{BijDefICMC}) is injective, suppose $\mathfrak{a}$ and $\mathfrak{b}$ are two fractional $\Z[\alpha]$-ideals in $K$ such that the matrices $A$ and $B$ for $m_\alpha$ with respect to some $\Z$-bases of $\mathfrak{a}$ and $\mathfrak{b}$ are conjugate in $M_n(\Z)$. We want to show $\mathfrak{a}$ and $\mathfrak{b}$ are in the same ideal class: $\mathfrak{b}=x\mathfrak{a}$ for some $x\in K^*$.\\
Since A represents $m_\alpha : \mathfrak{a} \rightarrow \mathfrak{a}$ with respect to some $\Z$-basis $\mathcal{A}$ of $\mathfrak{a}$, there is a commutative diagram
\[
\begin{tikzcd}
\mathfrak{a} \arrow{r}{[\cdot]_{\mathcal{A}}} \arrow[swap]{d}{m_\alpha} & \Z^n \arrow{d}{A} \\%
\mathfrak{a} \arrow{r}{[\cdot]_{\mathcal{A}}}& \Z^n
\end{tikzcd}
\]
where the horizontal arrows are the coordinate isomorphisms that identify $\mathcal{A}$ with the standard basis of $\Z^n$. Similarly for the basis $\mathcal{B}$ of $\mathfrak{b}$ with respect to which $m_\alpha : \mathfrak{b} \rightarrow \mathfrak{b}$ has matrix representation $B$, we have a commutative diagram
\[
\begin{tikzcd}
\mathfrak{b} \arrow{r}{[\cdot]_{\mathcal{B}}} \arrow[swap]{d}{m_\alpha} & \Z^n \arrow{d}{B} \\%
\mathfrak{b} \arrow{r}{[\cdot]_{\mathcal{B}}}& \Z^n
\end{tikzcd}
\]
Since $A$ and $B$ are conjugate in $M_n(\Z)$, $B = UAU^{-1}$for some $U\in GL_n(\Z)$, so
\[
\begin{tikzcd}
\Z^n \arrow{r}{U} \arrow[swap]{d}{A} & \Z^n \arrow{d}{B} \\%
\Z^n \arrow{r}{U}& \Z^n
\end{tikzcd}
\]
commutes. Let’s put these diagrams together:
\[
\begin{tikzcd}
\mathfrak{a} \arrow{r}{[\cdot]_{\mathcal{A}}} \arrow[swap]{d}{m_\alpha} & \Z^n \arrow{r}{U} \arrow[swap]{d}{A} & \Z^n \arrow{d}{B} \arrow{r}{[\cdot]_{\mathcal{B}}^{-1}} & \mathfrak{b} \arrow[swap]{d}{m_\alpha}\\%
\mathfrak{a} \arrow{r}{[\cdot]_{\mathcal{A}}} & \Z^n \arrow{r}{U}& \Z^n \arrow{r}{[\cdot]_{\mathcal{B}}^{-1}} & \mathfrak{b}
\end{tikzcd}
\]

Each square in the diagram commutes, so the whole diagram commutes. The top and bottom maps are $\Z$-linear isomorphisms, so the common composite map $\mathfrak{a}\rightarrow\mathfrak{b}$ on the top and bottom is a $\Z$-linear isomorphism that commutes with $m_\alpha$ by commutativity of the diagram around the boundary. That implies the composite map $\mathfrak{a}\rightarrow\mathfrak{b}$ is $\Z[\alpha]$-linear, not just $\Z$-linear, so $\mathfrak{a}$ and $\mathfrak{b}$ are isomorphic as $\Z[\alpha]$-modules. Isomorphic fractional $\Z[\alpha]$-ideals are scalar multiples of each other, so $\mathfrak{b}=x\mathfrak{a}$ for some $x\in K^*$. More specifically, any $\Z[\alpha]$-linear isomorphism of fractional $\Z[\alpha]$-ideals must be multiplication by some $x\in K^*$, so the composite map $\mathfrak{a}\rightarrow\mathfrak{b}$ along the top and bottom of the above commutative diagram is multiplication by $x$. This completes the proof of the theorem.
\end{proof}

\textbf{Remark.} There is a more conceptual proof of Theorem \ref{th:LatimerMacDuffee}, but the one given above is more explicit for our purposes.

We now extend the previous result to the cases interesting for us.
\begin{Prop}\label{prop:bij-ICM-ClM}
Let $f(T) \in \mathbb{Z}[T]$ be monic irreducible of degree $n$, and $\alpha$ is a root of $f(T)$, such that $f(0)=q^{g}$. Then we have bijections
\begin{enumerate}
    \item[i] $ \ICM(\Z[\alpha,q/\alpha])  \longleftrightarrow \{ A\in \mathfrak{Cl}(M_{n,f}(\Z)),\; q^{g-1}|\tau(A)\},$
    \item[ii] $ \ICM(\Z[\alpha,q/\alpha, \sigma_\ell])  \longleftrightarrow \{ A\in \mathfrak{Cl}(M_{n,f}(\Z)),\; q^{g-1}|\tau(A),\;\ell |\tau(1-A)\},$
\end{enumerate}
for a prime $\ell|f(1)$, and where $\sigma_\ell= \frac{f(1)}{\ell(1-\alpha)}$.
\end{Prop}

\begin{proof}
We consider the bijection given in the proof of Theorem \ref{th:LatimerMacDuffee}, as well as the same notations. A fractional $\ICM(Z[\alpha])$-ideal $\mathfrak{a}$ belongs to $\ICM(Z[\alpha, q/\alpha])$ if and only if $m_{q\alpha^{-1}}: \mathfrak{a}\rightarrow\mathfrak{a}$ is well defined; under the bijection \ref{BijDefICMC} (within the proof of Th. \ref{th:LatimerMacDuffee}), this corresponds to matrices $A\in M_{n,f}$ such that $qA^{-1}\in M_n(\Z)$. Then [i] follows from the relation $A^{-1}=\Cof(A)^{t}/\det(A)$ (where $\det(A)=q^{g}$) and the definition of $\tau$. The same for the [ii]
\begin{align*}
    \ICM(Z[\alpha, q/\alpha, f(1)/\ell(1-\alpha)]),
\end{align*}
by observing that $\det(1-A)=f(1)$.
\end{proof}

\subsection*{Cyclic varieties}
An abelian variety $A$ defined over $k$ is said to by \textbf{cyclic} if its groups $A(k)$ of rational points is cyclic. Otherwise, $A$ is said to be \textbf{not cyclic}.

For an elliptic curve $E$, we have that $E(\Fq)$ is not cyclic if and only if there exist a prime $\ell\neq p$ such that $E[\ell]\subseteq E(\Fq)$. This is not the case for higher dimensional varieties, however,  a slightly different criterion holds:
\begin{Lemma}\label{lemma:endom_criterion}
If $A$ is simple, then
\begin{align*}\label{EndCriterion}
A(\Fq) \text{ is not cyclic } \Longleftrightarrow \exists \varphi\in \End_{\Fq}(A), \; [f_A(1)/\ell] =\varphi\circ (1-F)
\end{align*}
for some prime divisor $\ell|f_A(1)$. If such $\varphi$ does exist, it must belong to the center of $\End_{\Fq}(A)$.
\end{Lemma}
\begin{proof}
This follows from the fact that $A(\Fq)$ is not cyclic if and only if $A(\Fq)=\ker (1-F) \subset \ker ([f_A(1)/\ell]) =A[f_A(1)/\ell]$ for some $\ell|f_A(1)$, where $F$ denotes the Frobenius endomorphism of $A$. From Theorem 4 in \cite{Mumford1999}, since $1-F$ is separable, $\ker (1-F) \subset \ker ([f_A(1)/\ell])$ if and only if such $\varphi$ does exist over $\overline{\mathbb{F}}_q$. Since $A$ is simple, if such $\varphi$ does exist, it must be in the center of $\End_{\overline{\mathbb{F}}_q}(A)$: if $\psi\in \End_{\overline{\mathbb{F}}_q}(A)$, multiplying the previous equation in both sides and subtracting them, we have $0=(\varphi\psi -\psi\varphi)\circ (1-F)$, thus $\varphi$ commutes with $\psi$. Using a similar argument, we can show that this endomorphism is in fact defined over $\overline{\mathbb{F}}_q$.
\end{proof}

Note that in this case we do not require $\ell\neq p$, however this criterion depends on the cardinality $f_A(1)$ of $A(\Fq)$. Lemma \ref{lemma:endom_criterion} implies that abelian varieties with few endomorphisms are more likely to be cyclic. The group structure of an abelian variety is not necessarily determined by its endomorphism ring, but from Lemma \ref{lemma:endom_criterion} it follows that the property of being cyclic or not depends on its endomorphism ring, and thus, it is an invariant of it.

Counting cyclic varieties in an ordinary isogeny class is the same as counting some fractional ideals in the center of $\End^0_{\Fq}(A)$.

\textbf{Remark.} In the non-ordinary case, we still can assign to every abelian variety $A$, an order $\mathcal{O}$, which represents the center of its endomorphism ring, and we can know if $A$ is cyclic or not by studying the order $\mathcal{O}$. When the base field is prime, there is a version of Deligne's functor given in \cite{centeleghe2015}. 

\subsection*{Proof of Theorem \ref{thm:bijection}}
The statement for $\mathfrak{m}_{f,1}$ follows immediately from the first bijection of the Proposition \ref{prop:bij-ICM-ClM} and the first bijection of Theorem \ref{Th:Deligne} (the Deligne's equivalence). Now, for $A\in\mathcal{A}$, denote by $I_A$ its associated fractional ideal (from Theorem \ref{Th:Deligne}), and by $[I_A]$ its associated class of matrices (from Theorem \ref{th:LatimerMacDuffee}), then using Lemma \ref{lemma:endom_criterion} we have:
\begin{align*}
A \text{ is not cyclic } &\Longleftrightarrow f(1)/\ell(1-F)\in \End_{\Fq} (A),\quad \ell|f(1) \text{ prime}\\
&\Longleftrightarrow (I_A:I_A)\supset Z[\alpha, q/\alpha, f(1)/\ell(1-\alpha)],\quad \ell|f(1) \text{ prime}\\
&\Longleftrightarrow I_A\in \ICM(Z[\alpha, q/\alpha, f(1)/\ell(1-\alpha)]),\quad \ell|f(1) \text{ prime}\\
&\Longleftrightarrow [I_A]\in \mathfrak{Cl}(M_{2g,f}(\mathbb{Z})),\; q^{g-1}|\tau([I_A]), (\tau(1-[I_A]),f(1))\geq 2,
\end{align*}
and the result follows.

\section*{Acknowledgement}
This work is part of my PhD Dissertation. I would like to thank my advisor Serge Vl\u{a}du\c{t} for very fruitful discussions and for his very useful comments and suggestions.

\bibliography{ajgiangreco-CAVFFOIC}
\bibliographystyle{siam}
\end{document}